\documentclass[reqno,11pt]{amsart}
\usepackage{amscd,amssymb, amsmath}
\usepackage{graphicx}
\usepackage[russian, english]{babel}

\newtheorem*{thm}{Theorem}
\newtheorem*{case}{Case}
\newtheorem{theorem}{Theorem}[section]
\newtheorem{lemma}[theorem]{Lemma}

\newtheorem{corollary}[theorem]{Corrolary}
\theoremstyle{definition}
\newtheorem{definition}[theorem]{Definition}

\newtheorem{question}[theorem]{Question}

\newtheorem{note}[theorem]{Note}

\begin{document}
\author{Lev Soukhanov}
\title{On the phenomena of constant curvature in the diffusion-orthogonal polynomials}
\maketitle
\begin{abstract}
We consider the systems of diffusion-orthogonal polynomials, defined in the work [1] of D. Bakry, S. Orevkov and M. Zani and (particularly) explain why these systems with boundary of maximal possible degree should always come from the group, generated by reflections. Our proof works for the dimensions $2$ (on which this phenomena was discovered) and $3$, and fails in the dimensions $4$ and higher, leaving the possibility of existence of diffusion-orthogonal systems related to the Einstein metrics.

The methods of our proof are algebraic / complex analytic in nature and based mainly on the consideration of the double covering of $\mathbb{C}^d$, branched in the boundary divisor.

Author wants to thank Stepan Orevkov, Misha Verbitsky and Dmitry Korb for useful discussions.
\end{abstract}

\footnote{
The author is partially supported by AG Laboratory HSE, RF government grant, ag.  11.G34.31.0023.}

\tableofcontents

\section{Introduction}
This paper is dedicated to the explanation of the strange correspondence discovered in the work [1], in which the following problem is considered: let $\Omega \in \mathbb{R}^d$ be a compact set, $\mu$ - integrable continuous measure on it, $L$ - elliptic operator of second order, adjoint w.r.t. $L^2(\Omega, \mu)$, preserving the space of polynomials of degree $\leq n$ for any $n$. These objects are called (multidimensional) diffusion orthogonal polynomials systems. As the operator $L$ is self-adjoint, it is of the form $L = \frac{1}{\mu}\partial_i g^{ij}\mu \partial_j$.

In [1], it is proved that the boundary is contained in the algebraic hypersurface $D = 0$, where $D = det(g^{ij})$. As coefficients of $g^{ij}$ are polynomials of degree at most $2$, maximal degree of $D$ is $2d$.

Classification problem is completely solved in dimension two, and it happens that for all the cases where Zariski closure of the boundary has the maximal possible degree $4$, the metric $g_{ij}$ is of constant nonnegative curvature (and the domains $\Omega$ can be identified with the fundamental domains of the reflection groups acting on the sphere or the Euclidean plane). Also, the measure in these cases is just $\mu = \sqrt{G}$, where $G = det(g_{ij}) = \frac{1}{D}$, and $L$ is just a standard Laplace-Beltrami operator for this metric.

\begin{thm} The main theorem of this paper states that these metrics are direct sums of Einstein metrics for any dimension with the assumptions from the work [1] (maximal degree of the boundary).
\end{thm}

We are unable to prove anything about "non-negative" part, and able to prove that angles of the domain are of the form $\frac{\pi}{n}, n \in \mathbb{N}$.

The paper is organised as follows: in the next section, we will state the fact from the complex algebraic geometry which we are going to prove and show how the core theorem follows from it. Then, we will define irreducible models and prove it for them, and then show that any model can be decomposed to the direct sum of irreducibles. Then, we present the theorem about angles, relation to the reflection groups and some more or less natural conjectures, and the last section is an appendix, devoted to the proof of the angles theorem.

\section{Reformulation of the statement}

Let $\Omega, g^{ij}, \mu = \sqrt{G}$ be as stated in the introduction and the work [1]. Zariski closure of $\partial \Omega$ is the divisor $D = 0$, and $\mu$ is integrable in $\Omega$, hence have no multiple components.

Then, Laplace-Beltrami operator is of the form

\begin{equation}
\Delta = g^{ij} \partial_i \partial_j + (\partial_j g^{ij} - -1/2 \frac{g^{ij} \partial_j D}{D}) \partial_i
\end{equation}

As in [1], we want it to have regular coefficients and to preserve the subspace of polynomials of degree $\leq n$ for any $n$. Hence, $deg(g^{ij}) \leq 2$.
We, as in introduction, assume $deg(D) = 2d$. From the fact that it has regular coefficients follows that $D$ divides $g^{ij} \partial_j D$ .

So, our setting will be the following:

Let $g^{ij}$ be the polynomial symmetric tensor with matrix components of degree $\leq 2$ on $\mathbb{C}^d$, $D = det(g^{ij})$, degree of $D$ is $2d$ (maximal possible), $D$ has no multiple components and the condition

\begin{equation}
(g^{ij} \partial_j D) { \vdots }{ D }
\end{equation}

holds.

\begin{theorem}
It is the direct sum of metrics with $Ric(g) = \lambda g$ (maybe, with different $\lambda$'s).
\end{theorem}

\section{Proof for irreducible models}

\lemma {Let $S$ be a double covering of $\mathbb{C}^d$ in $D = 0$ (let us recall that $D = 0$ has no components of multiplicity higher than $1$), and let us denote the standard projection of $S$ as $\pi$. Then $\pi^* g_{ij}$ is regular and non-degenerate in codimension 2 (in the smooth points of double covering).}
\proof
We shall prove it analytically from the condition (2). Without loss of generality we can assume that $0$ is the smooth point of $D = 0$ and use local holomorphic coordinates with $x_0 = D$.

Then local coordinates in $S$ are $y_0^2 = x_0$, $y_i = x_i$ for $i > 0$.

$g^{0j}\partial_0 x_0 = g^{0j}$ is divisible by $x_0$, hence by $y_0 ^2$. Also, from $det(g^{ij}) = x_0$ $g^{00}$ is not divisible by $x_0^2$

$\frac{\partial}{\partial x_0} = \frac{-1}{2y_0}\frac{\partial}{\partial y_0}$

So, for $i>0$ $g_{0i}$ will have strictly positive valuation by $y_0$ and $g_{00}$ will have zero valuation.

\lemma Conversely, if there is a cometric $n^{ij}$ on the smooth part of $S$ which is invariant under the involution $y_0 \rightarrow -y_0$, it comes from the cometric on $\mathbb{C}^d$

\proof Its invariance and the fact that $y_0^2 = x_0$ guarantees that it is sum of $A^{ij}\frac{\partial}{\partial x_i}\frac{\partial}{\partial x_j} + B^i y_0 \frac{\partial}{\partial x_i}\frac{\partial}{\partial y_0} + C (\frac{\partial}{\partial y_0})^2$ which equals $A^{ij}\frac{\partial}{\partial x_i}\frac{\partial}{\partial x_j} - 2B^i x_0 \frac{\partial}{\partial x_i} \frac{\partial}{\partial x_0} + 4C y_0(\frac{\partial}{\partial y_0})^2$, where $A, B, C$ are even in $y_0$ and hence depend of $x_0$. Now it is proven in the codimension $2$ and hence, by Hartogs' extension theorem this cometric is regular in $\mathbb{C}^d$

\begin{definition} $g^{ij}$ is called irreducible if for any $\tilde{g}^{ij}$ of degree $\leq 2$ such that it comes from the double covering we have $\tilde{g}^{ij} = \lambda g^{ij}$ \end{definition}

\note For this definition it is not obvious that any model can be decomposed into the direct sum of irreducibles, but the direct sum is, obviously, non-irreducible (because we can take the direct sum of metrics multiplied by different constants).

\lemma $R^{ij} = g^{ik}g^{jl}R_{kl}$ is regular, with matrix components of degree $\leq 2$, where $R_{kl}$ is standard Ricci curvature tensor.
\proof Regularity follows from the fact that the metric is regular and non-degenerate on the double covering (hence, Ricci tensor is regular) and lemma 3.2. Now we need only to check the degree of the matrix components as rational functions, and the fact that they are regular is proven already, so they will be polynomials of degree $\leq 2$.

$deg (g_{ij}) \leq -2$ (here we crucially use $deg(D) = 2d$)

$\Gamma_{ij}^k = g^{sk}\frac{1}{2}(\partial_s g_{ij} - \partial_i g_{js} - \partial_j g_{is})$

$deg(\Gamma_{ij}^k) \leq -1$

$R_{ijk}^l = \partial_i \Gamma_{jk}^l - \partial_j \Gamma_{ik}^l + \Gamma_{is}^l \Gamma{jk}^s - \Gamma_{js}^l \Gamma_{ik}^s$

$deg(R_{ijk}^l) \leq -2$

$R_{ij} = R_{isj}^s$

$deg(R_{ij}) \leq -2$

$R^{ij} = g^{ik}g^{jl}R_{kl}$

$deg(R^{ij}) \leq 2$

\begin{theorem}
Hence, $R^{ij}$ is proportional to $g^{ij}$ for an irreducible $g^{ij}$.
\end{theorem}

\section{Decomposition into irreducibles}

\lemma Let $A \in End(TS)$ be an operator, invariant under the involution. Then its direct image on $\mathbb{C}^d$ is regular.
\proof Let us prove that it is regular in the smooth points of $D = 0$ and then general fact will follow by Hartogs' extension theorem. Local formula (in the notation of the lemma 3.1) for $A$ is $A^i_j \frac{\partial}{\partial y_i} \otimes dy_j$, and as it is invariant under $y_0 \rightarrow -y_0$, $A^0_0$ and $A^i_j$ for $i, j \neq 0$ are even in $y_0$, and $A^i_0, A^0_j$ for $i, j \neq 0$ are odd in $y_0$.

Let us recall that $dy_i = dx_i, \frac{\partial}{\partial y_i} = \frac{\partial}{\partial x_i}$ for $i \neq 0$,

$dx_0 = 2y_0 dy_0, \frac{\partial}{\partial x_0} = \frac{1}{2y_0} \frac{\partial}{\partial y_0}$

So, $\tilde{A^i_j} \frac{\partial}{\partial x_i} \otimes dx_j = A^i_j \frac{\partial}{\partial y_i} \otimes dy_j$ is well defined, regular ($\tilde{A^i_j}$ is constructed by dividing the first row of $A^i_j$ by $2y_0$ and multiplying the first column by the same, but the elements $A^i_0$ are odd in $y_0$, and, hence, they can be divided), with all matrix components even in $y_0$, hence depending on $x_0, ..., x_n$.

\lemma Let $g^{ij}$ be the non-irreducible cometric. Then it can be decomposed (after some linear change of coordinates) as $g^{ij} = a^{ij} + b^{ij}$ where $a^{ij} = 0$ for $i, j \leq k$, $b^{ij} = 0$ for $i, j > k$.

\proof Let $s^{ij}$ be another cometric, coming from the double covering in $D$. Let us consider an operator $A^i_j = s^{ik}g_{kj}$. It is polynomial by the lemma 4.1, and its coefficients are rational fuctions of degree $\leq 0$, so, in fact $A^i_j$ is the symmetric operator with constant coefficients. Let us linearly change the coordinates in such a way that $V = \langle e_0 ... e_k \rangle$, $W = \langle e_{k+1} ... e_{d-1}\rangle$ are eigenspaces of $A$. Then, these subspaces are orthogonal with respect to $g^{ij}$ and hence $g^{ij}$ has a block matrix $a^{ij} \oplus b^{ij}$ (however, $a$ and $b$ could, a priori, depend on all the variables $x_0 ... x_{d-1}$).

From now on cometric $g^{ij}$ is considered to be $a^{ij} \oplus b^{ij}$

\lemma For an any smooth point of $D = 0$ either $V$ or $W$ lies in a tangent space.
\proof Let us consider the double covering. $V$ and $W$ are well-defined on it, and in a smooth point of ramification divisor one of these spaces is invariant under the differential of an involution (as it has only one eigenvalue $-1$).

So, $D = 0$ can be decomposed into the two components, one of them is independent on $x_1, ..., x_k$ and another independent on $x_{k+1}, ..., x_d$.

\begin{theorem}
Any model is the direct sum of irreducible components.
\end{theorem}

\begin{proof}
From lemmas 4.2 and 4.3 we have the following situation: $D = D_a D_b$, where $D_a$ depends only on $x_1, ..., x_k$, $D_b$ only on $x_{k+1}, ..., x_d$. Let us consider $a^{ij}$ as a cometric depending of $x_1, ..., x_k$, and consider $x_{k+1}, ..., x_d$ as parameters. For an any fixed tuple of parameters $a^{ij}$ is an admissible cometric with the determinant $D_a$. By induction, we can assume that it is the direct sum of irreducibles $a^{ij} = \bigoplus_t \lambda_t a^{ij}_t$. Only way $a^{ij}$ could depend on parameters is by varying $\lambda$'s. But it is impossible because each $a^{ij}_t$ has degree $2$, and if $\lambda_t$ varied then $a^{ij}$ would have degree at least $3$. Hence, $a^{ij}$ depends only on $x_1, ..., x_k$. The $b^{ij}$, correspondingly, depends only on $x_{k+1}, ... x_d$.
\end{proof}

Now, the theorems 3.6 and 4.4 give us the main statement.

\section{Consequences and conjectures}

Constant curvature models in dimension two are related to groups, generated by reflections. Let us describe briefly how does it happen.

At first, the domain $\Omega$ can be identified with the subset of the space of constant curvature, and the smooth parts of the boundary are geodesic (because the surface $S$ has the involution, which preserves the ramification divisor). The angles of the polygon are always $\frac {\pi}{n}$, and the corresponding singularities of boundary equation (locally) have the form $y^2 = x^n$. It is not strange, as the double covering along such a curve is a famous Kleinian singularity $\mathbb{C}^2 / \frac{\mathbb{Z}}{n\mathbb{Z}}$.

Let us also note that in the dimension $3$ constant Ricci curvature is equivalent to the constant sectional curvature. So, it maybe worth to ask, whether all the models of dimension $3$ has dihedral angles of form $\frac{\pi}{n}$, hence, come from the fundamental domains of reflection groups.

\begin{question}
In what generality or with what additional requirements we could prove the following statement: "Let $S$ be a surface with isolated singularities, endowed with holomorphic non-degenerate metric, defined on a smooth locus of a surface. Then \textit{under additional requirements?} all singularities are Kleinian of type $A_n$".
\end{question}

\begin{theorem}
Let $g^{ij}$ the germ of cometric on $\mathbb{R}^2$ in a point $0$, $\Omega$ be the germ of the closed set such that $\partial \Omega \subseteq \{x \in \mathbb{R}^2|D(x) = 0\}, \sqrt{G}$ is locally integrable in $\Omega$. Then the double covering branching in the divisor $D = 0$ is a Kleinian singularity of type $A_n$
\end{theorem}
\begin{proof} in appendix.
\end{proof}

\begin{note}
Let us consider the $d$-dimensional model. Locally around the edge of codimension $2$ we can consider our singularity as the family of singularities of dimension $2$. The proof of the theorem 5.2 works for such a families, because the integrability conditions of the model (no components with multiplicity higher than $1$ on the blow-up) implies the same conditions on each model of the family. So all possible singularities of codimension $2$ are $A_n \times \mathbb{C}^{d-2}$
\end{note}

\begin{corollary}
Dihedral angles of $\Omega$ are of the form $\frac{\pi}{n}$.
\end{corollary}
\begin{proof}
Consider the $2$-dimensional plane, which intersects the angle transversally. It is orthogonal w.r.t to our riemannian metric to the smooth part of the boundary, because its preimage under the double covering is smooth and invariant under the reflection in this smooth part (involution of the double covering). We can restrict the metric on this plane and prove the fact only for dimension $2$, and by the note 5.3 it will be the singularity of type $A_n$.

Let us consider the singular point $x$ on the boundary. Let us consider the double covering of its neighborhood in $\mathbb{C}^2$. It is endowed with the holomorphic metric, non-degenerate in the smooth locus. By the theorem 5.2 it can be identified with the quotient of the open neighborhood of the point $0 \in U \subset \mathbb{C}^2$ by the group $\frac{\mathbb{Z}}{n\mathbb{Z}}$, acting by the standard rotations. The pullback of the holomorphic metric on $U$ is well-defined and non-degenerate everywhere by the Hartogs' theorem. Hence, linearization of the metric in $x$ is $\mathbb{C}^2 / \frac{\mathbb{Z}}{n\mathbb{Z}}$ with the standard flat metrics on it.
\end{proof}

\begin{theorem}
Any model of constant positive sectional curvature can be identified with the quotient $S^d / G$, where $G$ is the group, generated by reflections.
\end{theorem}
\begin{proof}
As a riemannian manifold with boundary, the polyhedra of constant curvature with dihedral angles of $\frac{\pi}{n}$ can be identified with the domain of group $G$, generated by reflections by Poincare fundamental polyhedron theorem. The real algebraic structure can be recovered from the riemannian structure as follows: the smooth function is called regular if and only if it is the finite sum of Laplacian eigenfunctions (because we required that Laplace-Beltrami operator eigenfunctions are polynomials), and the algebraic structure on $S^d / G$ can be recovered in the same way (eigenvectors of spherical Laplacian are spherical harmonics, and spherical harmonics are regular on $S^d$).
\end{proof}

\begin{note}
Let us note that the theorem 5.2 considers only codimension $2$ singularities on the real part of the divisor $D = 0$, and also only singularities contained in $\partial \Omega$, and the theorem 5.5 completely ignores all other singularities, because it uses only riemannian structure of $\Omega$, hence, post factum it guarantees absence of such singularities.
\end{note}

\begin{question} We didn't prove anything about positivity of the curvature. Is it true and if it is, in what generality?
\end{question}

Expanding the previous question, let us consider the upper half plane $H$ with the cocompact group $\Gamma$, acting on it, generated by reflections. We can try to turn $H / \Gamma$ into algebraic variety with the following construction. Let us take the $\Gamma$-invariant eigenfunctions of Laplacian. If their linear span (without any topological structure) is closed under the multiplication, we can take it as a commutative algebra of functions on $H / \Gamma$.

\begin{question} Is it true? If yes, is this algebra finitely generated and when is it isomorphic to $\mathbb{C}^2$? If yes, what kind of condition should hold for the cometric on $\mathbb{C}^2$ to make these cases to be a part of classification?
\end{question}

\begin{question} Does there exist a model in dimension $4$ which is not the direct sum of metrics of constant curvature? If yes, what kind of polynomial series arise from it?
\end{question}

\section{Appendix: the proof of the theorem 5.2}

This appendix is more or less independent of the other parts of the paper and consists of considerations about integrability of volume forms in a style of real log-canonical threshold theory from [2], however, we did not use any results from [2] directly.

\begin{definition}
The singularity is called du Val, if it is of the form $\mathbb{C}^2 / G$, where $G$ is a finite subgroup of $SL_2(\mathbb{C})$.
\end{definition}

\begin{definition}
The singularity is called symplectic, if it has a resolution which is a holomorphically-symplectic variety.
\end{definition}

\begin{thm}[Artin, 1966]
Every symplectic singularity of dimension $2$ is du Val singularity.
\end{thm}

\begin{proof}
It is the main result of the work [3].
\end{proof}

\begin{lemma}
Only du Val singularities which admit non-degenerate holomorphic riemannian metric on their smooth loci are the singularities of $A_n$ type (i.e. quotients of $\mathbb{C}^2$ by the rotation group $\frac{\mathbb{Z}}{n\mathbb{Z}}$). Also, they admit only one real structure up to isomorphism.
\end{lemma}

\begin{proof}
Consider the covering of the singularity with $\mathbb{C}^2$. The pullback of the riemannian metric is regular and non-degenerate in $0$ because of Hartogs' theorem. Then let us consider the action of our group on the tangent space to $0$ (linearize it). The group which preserves riemannian metric and volume form on $\mathbb{C}^2$ is the complexification of the group of rotations $\mathbb{C}^*$, and only finite subgroups of it are $\frac{\mathbb{Z}}{n\mathbb{Z}}$.
\end{proof}

\begin{thm}[Toy theorem]
Let us consider the real-algebraic cometric $g^{ij}$ on $X = \mathbb{R}^2$ such that $\sqrt{|G|}$ is locally integrable and satisfying the condition (2). Then singularities of the real part of the double covering, branched in the divisor $D = 0$ are symplectic, hence, du Val and, hence, $A_n$.
\end{thm}

\begin{proof}
First of all, $D$ has no multiple components, otherwise $\sqrt{|G|}$ it is not locally integrable around the smooth point of multiple component. We are now going to consider the resolution of the singularity point and keep track of multiplicites of vanishing of the tensor field $D(x, y) (\frac{\partial}{\partial x} \wedge \frac{\partial}{\partial y})^{\otimes 2} \in \Gamma(K^{-2})$ while we are blowing up singular points.

The multiplicity of the pullback after the blow-up is the multiplicity of the singularity (it is always strictly greater than $1$, as we never blow up a smooth point) minus $2$, because relative canonical class is an exceptional curve with multiplicity $1$, and we consider the section of $K_X^{-2}$. So, the multiplicity of the exceptional curve after the blow-up is always non-negative. If it is at least $2$, $\sqrt{|G|}$ is non-integrable, hence it must be $0$ or $1$.

So, after blowing everything up to the SNC divisor we get the collection of normally crossing rational curves with multiplicity $1$. Let us consider the blow-ups in the points of normal crossings: then the pullback will have multiplicity $0$ on the exceptional curves, and so the resulting curves with multiplicity $1$ won't intersect at all. Let us call this resolution $\tilde{X}$, and let us denote the projection by $\pi$.

Let us consider the double coverings
\[S = \{(p, v), p \in X, v \in \Lambda^2 T_p (X)| v^2 = D|_p\} \]
\[ \tilde{S} = \{(p, v), p \in \tilde{X}, v \in \Lambda^2 T_p (\tilde{X})| v^2 = \pi^*D|_p\}\]

$\tilde{S}$ is clearly a resolution of $S$. But $\tilde{S}$ endowed with non-vanishing regular 2-form $v^{-1}$ (it regularises on the branching divisor, analogously to the proof of lemma 3.1), hence, it is symplectic resolution $S$. Now, by the theorem of Artin it is du Val, and by the lemma 6.3 it is $A_n$
\end{proof}

Now we are going to significantly weaken assumption of the theorem.

\begin{thm}[5.2]
Let $g^{ij}$ be as in the Toy Theorem, $\Omega$ be the closure of one of the connected components of $\mathbb{R}^2 \\ \{x| D(x) = 0\}$, and let us weaken the integrability condition to be ''$\sqrt{G}$ is integrable in $\Omega$''. Then all the singular points of $D = 0$ contained in $\Omega$ are of the $A_n$ type.
\end{thm}

\begin{proof}

Before starting the proof let us fix the terminology. We will call the component of $det(g^{ij}) \in \Gamma(K_X^{-2})$ the \textbf{boundary component} if it is contained in the analytic closure of the boundary, the \textbf{adjacent component} if it intersects the boundary (of course, in the singular point) and \textbf{nonadjacent component} if it doesn't intersect $\Omega$ at all. Our convention will be that when we do a blow-up we will denote the proper preimages of the components by the same letters by which we denoted the components, and the proper preimage of $\Omega$ is defined as the closure of the preimage of the interior of the $\Omega$.

\begin{note}
In our setting, Zariski closure of $\partial \Omega$ is $D = 0$, however, for the local analytic closure around the singularity it is not obvious and, also, the situation might change after the blow-up. Before doing any blow-ups we doesn't have any nonadjacent components, because our considerations are local around the singularity.
\end{note}

\begin{note}
Boundary components of multiplicity higher than $1$ are impossible due to integrability conditions. Adjacent components might become nonadjacent after a blow-up (and become nonadjacent in the SNC resolution), but they do affect the multiplicities on the exceptional divisor.
\end{note}

\begin{note}
Exceptional divisor of the blow-up of the singularity is always adjacent or boundary. To prove it consider the smooth curve in interior of $\Omega$ which tends to the singularity, and take its proper preimage.
\end{note}

We always proceed by induction by the number of blow-ups needed to resolve the singularity.

\begin{lemma}
Adjacent component with multiplicity higher than $1$ leads to the non-integrability of $\sqrt{G}$.
\end{lemma}
\begin{proof}
 Let us denote this adjacent component as $l$, its multiplicity as $m$, and all other components as $\omega$.
 The first case is that $ord_\omega$ or $ord_l$ in the singularity is higher than one. Then the exceptional divisor will have multiplicity higher than $1$, it is adjacent and we proceed by induction.
 The second case is that $ord_\omega = ord_l = 1$, and it means that $l$ and $\omega$ are smooth, and our singularity is just the multiple tangent, so we might assume $l$ to be $y = 0$, $\omega$ to be $y = x^{2n}$, $\Omega$ to be $y \geq x^{2n}$. We blow-up the singularity and denote the exceptional curve by $e$, it has the multiplicity $m-1$. $e$ is boundary. Let us blow-up the intersection point of $e$ and $\omega$ and denote the exceptional curve as $e'$. Then $e'$ is boundary and has multiplicity $m$. As $m > 1$ it contradicts integrability.
\end{proof}

So, the order of singularity $\leq 3$ (because exceptional divisor is always boundary or adjacent, and multiplicity of exceptional divisor is $ord_D - 2$) and it has no multiplicities.

We are going to check that after the blow-up all nonadjacent components are of $A_n$ type and proceed by induction - then we are guaranteed to get the resolution from the Toy Theorem, and the rest of the proof is the same.

\begin{case} Singularity is of order $2$.
\end{case}
Then it is of type $A_n$, process terminated.

\begin{case} Singularity is of order $3$, no adjacent components.
\end{case}
No nonadjacent components in the resolution.

\begin{case} Adjacent component is of order $2$, and boundary is of order $1$.
\end{case}
Changing the coordinates we can assume that the boundary is a line, and $\Omega$ is a half-plane. Then, after the blow-up, any point of exceptional curve lies in $\partial \Omega$, hence adjacent component won't become nonadjacent.

\begin{case} Adjacent component is of order $1$, and boundary is of order $2$.
\end{case}
By the change of coordinates we can assume that adjacent component is a line. If the adjacent component becomes nonadjacent after the blow-up then it intersects only the exceptional curve, and transversally ($A_2$ singularity), condition checked.

We are done - we blow up to the SNC divisor, don't meet any adjacent multiple components by the lemma 6.7, nonadjacent are always of $A_n$ type so when we resolve them we don't meet multiple components, too, and then we just use the proof of the Toy Theorem.
\end{proof}

\end{document}